\documentclass[12pt, twoside]{article}
\usepackage{amsmath,amsthm,amssymb}
\usepackage{times}
\usepackage{enumerate}
\usepackage{bm}
\usepackage[all]{xy}
\usepackage{mathrsfs}
\usepackage{amscd}
\usepackage{color}
\def\titlerunning#1{\gdef\titrun{#1}}
\makeatletter
\def\author#1{\gdef\autrun{\def\and{\unskip, }#1}\gdef\@author{#1}}
\def\address#1{{\def\and{\\\hspace*{18pt}}\renewcommand{\thefootnote}{}%
		\footnote {#1}}%
	\markboth{\autrun}{\titrun}}
\makeatother
\def\email#1{e-mail: #1}

\def\keywords#1{\par\medskip
	\noindent\textbf{Keywords.} #1}

\newtheorem{theorem}{Theorem}[section]
\newtheorem{corollary}[theorem]{Corollary}
\newtheorem{lemma}[theorem]{Lemma}
\newtheorem{proposition}[theorem]{Proposition}
\theoremstyle{definition}
\newtheorem{definition}[theorem]{Definition}
\newtheorem{remark}[theorem]{Remark}

\theoremstyle{example}
\newtheorem{example}[theorem]{Example}
\numberwithin{equation}{section}

\xyoption{all}

\def \C {\mathbb{C}}

\def \a {\alpha }
\def \be {\beta}

\def\w {\omega}
\def\Om{\Omega}

\def\Ga{\Gamma}

\begin{document}
\baselineskip=17pt

\titlerunning{A note on Euler number of locally conformally K\"{a}hler manifolds}
\title{A note on Euler number of locally conformally K\"{a}hler manifolds}

\author{Teng Huang}

\date{}

\maketitle

\address{T. Huang: School of Mathematical Sciences, University of Science and Technology of China; Key Laboratory of Wu Wen-Tsun Mathematics, Chinese Academy of Sciences, Hefei, 230026, P.R. China; \email{htmath@ustc.edu.cn; htustc@gmail.com}}
\begin{abstract}
Let $M^{2n}$ be a compact Riemannian manifold of non-positive (resp. negative) sectional curvature. We call $(M,J,\theta)$ a $d$(bounded) locally conformally K\"{a}hler manifold if the lifted Lee form $\tilde{\theta}$ on the universal covering space of $M$ is $d$(bounded). We show that if $M^{2n}$ is homeomorphic to a $d$(bounded) LCK manifold, then its Euler number satisfies the inequality $(-1)^{n}\chi(M^{2n})\geq$ (resp. $>$) $0$. 
\end{abstract} 
\section{Introduction}
In the early 19th century, Hopf proposed the following conjecture: \\
``Let $(M^{2n},g)$ be a closed $2n$-dimensional Riemannian manifold. If the sectional curvature of $M^{2n}$ is negative, then the Euler number $\chi(M^{2n})$ of $M^{2n}$  satisfies the inequality $(-1)^{n}\chi(M^{2n})>0$.'' \\
This conjecture is true in dimensions $2$ and $4$ \cite{Chern}. In dimension 2, by the Gauss-Bonnete formula, one can see that a closed Riemannian surface with negative section curvature has negative Euler number. In dimension 4, Chern \cite{Chern} proved that negative sectional curvature implies that Gauss–Bonnet integrand is pointwise positive.

Dodziuk \cite{Dodziuk} has proposed to settle the Hopf conjecture using the Atiyah index theorem for coverings (see \cite{Atiyah}). In this approach, one is required to prove a vanishing theorem for $L^{2}$ harmonic $k$-forms, $k\neq n$, on the universal covering of $M^{2n}$. The vanishing of these $L^{2}$ Betti numbers implies, by Atiyah's result, that $(-1)^{n}\chi(M^{2n})\geq0$. The strict inequality $(-1)^{n}\chi(M^{2n})>0$ follows provided one can establish the existence of nontrivial $L^{2}$ harmonic $n$-forms on the universal cover.  The program outlined above was carried out by Gromov \cite{Gromov} when the manifold in question is K\"{a}hler and is homotopy equivalent to a compact manifold with strictly negative sectional curvatures. The center idea is that Gromov introduced the notion of K\"{a}hler hyperbolicity. Gromov \cite{Gromov} points out that a bounded closed $k$-form, $k\geq2$, on a complete simply-connected manifold whose sectional curvatures are bounded above by a negative constant is automatically $d$(bounded) then he proved the above conjecture in the K\"{a}hler case. In order to attack Hopf Conjecture in the K\"{a}hlerian case when $K\leq0$ by extending Gromov's idea, Cao-Xavier \cite{CX} and Jost-Zuo \cite{JZ} independently introduced the concept of K\"{a}hler non-ellipticity, which includes nonpositively curved compact K\"{a}hler manifolds, and showed that their Euler characteristics satisfies $(-1)^{n}\chi(M^{2n})\geq0$. 

Inspired by K\"{a}hler geometry, Tan-Wang-Zhou gave the definition of symplectic parabolic manifold \cite{TWZ}.  By a well known result that a closed symplectic manifold satisfies the hard Lefschetz property if only if de Rham cohomology consists with the new symplectic cohomology, they proved that if $(M,\w)$ is a $2n$-dimensional closed symplectic parabolic manifold which satisfies the hard Lefschetz property, then the Euler number satisfies $(-1)^{n}\chi(M^{2n})\geq0$.

Let $(M,J,g)$ be a Hermitian manifold of complex dimension $n$, where $J$ denotes its complex structure, and $g$ its Hermitian metric. There is a interesting structure in $M$. We call $g$ a locally conformal K\"{a}hler  metric if $g$ is conformal to some local K\"{a}hlerian metric in the neighborhood of each point of $M^{2n}$. In many situations, the LCK structure becomes useful for the study of topology in non-k\"{a}hler geometry \cite{OV1,OV2}. The aim of this paper is to establish its validity for all $n$ in the $d$(bounded) locally conformally K\"{a}hler manifold, see Definition \ref{D1}.
\begin{theorem}\label{T6}
	Let $M^{2n}$ be a compact $2n$-dimensional Riemannian manifold of non-positive (resp. negative) curvature. If $M^{2n}$ is homeomorphic to a $d$(bounded) locally conformally K\"{a}hler manifold, then the Euler number of $M^{2n}$ satisfies the inequality $(-1)^{n}\chi(M^{2n})\geq$ (resp. $>$) $0$.	
\end{theorem}

\section{Preliminaries}
\subsection{$L^{2}$-Hodge number}
Let $(M,g)$ be a complete Riemannian manifold. Let $\Om^{k}(M)$ and $\Om^{k}_{0}(M)$ denote the smooth $k$-forms on $M$ and the smooth $k$-forms with compact support on $M$. Let $\langle\cdot,\cdot\rangle$ denote the pointwise inner product on $\Om^{\ast}(M)$ given by $g$ and duality. The global inner product is defined by
$$(\a,\be)=\int_{M}\langle\a,\be\rangle d{\rm{Vol}}_{g}.$$
We also write $|\a|^{2}=\langle\a,\a\rangle$, $\|\a|\|_{L^{2}(M,g)}=\int_{M}|\a|^{2}d\rm{Vol}_{g}$ and let
$$\Om^{k}_{(2)}(M,g)=\{\a\in\Om^{k}(M):\|\a\|_{L^{2}(M,g)}<\infty\}.$$
We denote by $\mathcal{H}^{k}_{(2)}(M,g)$ its space of $L^{2}$-harmonic $k$-forms, that is to say the space of $L^{2}$ $k$-forms which are closed and co-closed:
$$\mathcal{H}_{(2)}^{k}(M,g)=\{\a\in\Om^{k}_{(2)}(M,g):d\a=d^{\ast}\a=0 \},$$
where $$d:\Om^{k}_{0}(M)\rightarrow\Om^{k+1}_{0}(M)$$ is the exterior differential operator and 
$$d^{\ast}:\Om_{0}^{k+1}(M)\rightarrow\Om_{0}^{k}(M)$$
its formal adjoint. The operator $d$ does not depend on $g$ but $d^{\ast}$ does.

We assume throughout this subsection that $(M,g,J)$ is a compact complex $n$-dimensional manifold with a Hermitian metric $g$, and $\pi:(\tilde{M},\tilde{g},\tilde{J})\rightarrow(M,g,J)$ its universal covering with $\Gamma$ as an isometric group of deck transformations. Let $\mathcal{H}^{k}_{(2)}(M)$ be the spaces of $L^{2}$-harmonic $k$-forms on $\Om^{p,q}_{(2)}(\tilde{M})$, the squared integrable $k$-forms on $(\tilde{M},\tilde{g})$, and denote by $\dim_{\Gamma}\mathcal{H}^{k}_{(2)}(\tilde{M})$ the Von Neumann dimension of $\mathcal{H}^{k}_{(2)}(\tilde{M})$ with respect to $\Gamma$ \cite{Atiyah,Pansu}. Its precise definition is not important in our article but only the
following two basic facts are needed, see \cite{Gromov,Pansu}.\\
(1)  $\dim_{\Ga}\mathcal{H}_{(2)}^{k}(M)=0 \Leftrightarrow \mathcal{H}_{(2)}^{k}(M)=\{0\},$\\
(2) $\dim_{\Ga}\mathcal{H}$ is additive. Given $$0\rightarrow\mathcal{H}_{1}\rightarrow\mathcal{H}_{2}\rightarrow \mathcal{H}_{3}\rightarrow 0,$$
one have $$\dim_{\Ga}\mathcal{H}_{2}=\dim_{\Ga}\mathcal{H}_{1}+\dim_{\Ga}\mathcal{H}_{3}.$$
We denote by $h_{(2)}^{k}(M)$ the $L^{2}$-Hodge numbers of $M$, which are defined to be $$h_{(2)}^{k}(M):=\dim_{\Gamma}\mathcal{H}_{(2)}^{k}(\tilde{M}),\ (0\leq p,q\leq n).$$ 
It turns out that $h^{k}_{(2)}(M)$ are independent of the Hermitian metric $g$ and depend only on $(M,J)$. By the $L^{2}$-index theorem of Atiyah \cite{Atiyah}, we have the following crucial identities between $\chi(M)$ and the $L^{2}$-Hodge numbers $h_{(2)}^{k}(M)$: $$\chi(M)=\sum_{k=0}^{n}(-1)^{k}h_{(2)}^{k}(M).$$
\begin{remark}
	$L^{2}$-Betti number are not homotopy invariants for complete non-compact manifolds. But Dodziuk \cite{Dod} proves that the class of the representation of $\Gamma$ on the space of $L^{2}$-harmonic forms is a homotopy invariant of $M$. In particular the $\Gamma$-dimension (in the sense of Von Neumann) of the space of $L^{2}$-harmonic forms does not depend on the chosen $\Gamma$-invariant metric.
\end{remark}
\subsection{K\"{a}hler parabolic}
Let $(M,g)$ be a Riemannian manifold. A differential form $\a$ is called $d$(bounded) if there exists a form $\be$ on $M$ such that $\a=d\be$ and 
$$\|\be\|_{L^{\infty}(M,g)}=\sup_{x\in M}|\be(x)|_{g}<\infty.$$
It is obvious that if $M$ is compact, then every exact form is $d$(bounded). However, when $M$ is not compact, there exist smooth differential forms which are exact but not $d$(bounded). For instance, on $\mathbb{R}^{n}$, $\a=dx^{1}\wedge\cdots\wedge dx^{n}$ is exact, but it is not $d$(bounded). 

Let’s recall some concepts introduced in \cite{CX,JZ}. A differential form $\a$ on a complete non-compact Riemannian manifold $(M,g)$ is called $d$(sublinear) if there exist a differential form $\be$ and a number $c>0$ such that $\a=d\be$ and
$$ \ |\a(x)|_{g}\leq c,$$
$$ |\be(x)|_{g}\leq c(1+\rho_{g}(x,x_{0})),$$ 
where $\rho_{g}(x,x_{0})$ stands for the Riemannian distance between $x$ and a base point $x_{0}$ with respect to $g$.

The concept of $d$(sublinear) is both natural and flexible. We recall the following classical fact pointed out by Gromov \cite{CY,Atiyah} (negative case) and Cao-Xavier \cite{CX} (non-positive case).
\begin{theorem}\label{T3}
Let $M$ be a complete simply-connected manifold of non-positive (resp. negative) sectional curvature and $\a$ a bounded closed $k$-form on $M$. Then $\a$ is $d$(sublinear) for $k\geq 1$ (resp. $d$(bounded) for $k\geq2$).
\end{theorem}
\begin{proposition}
	If $(M,g)$ is a closed smooth Riemannian manifold of non-positive (resp. negative) sectional curvature and $\pi:(\tilde{M},\tilde{g})\rightarrow (M,g)$ its universal covering. If $\a$ is a closed $k$-form on $M$, then the lifted $k$-form $\tilde{\a}:=\pi^{\ast}\a$ is $d$(sublinear) for $k\geq 1$ (resp. $d$(bounded) for $k\geq2$).
\end{proposition}
The following results are the main theorems in \cite{CX,Gromov,JZ}.
\begin{theorem}\label{T5}
	Let $(M,g)$ be a complete $2n$-dimensional K\"{a}hler manifold with a $d$(sublinear) K\"{a}hler form $\w$. Then $\mathcal{H}^{k}_{(2)}(M,g)\neq\{0\}$,when $k\neq n$, i.e., 
	$$h_{(2)}^{k}(M)=0,\ k\neq n.$$ 
	Furthermore, if the K\"{a}hler form $\w$ is $d$(bounded), then $\mathcal{H}^{n}_{(2)}(M,g)\neq\{0\}$, i.e,
	$$h_{(2)}^{n}(M)\geq 1.$$
\end{theorem}
\begin{remark}
	The proof for the vanishing type results in all cases is a direct application of the $L^{2}$ version’s Lefschetz theorem.  The real hard part is the nonvanishing results in negative case, where a careful analysis on the lower bound of the eigenvalues of the Laplacian on $L^{2}$-harmonic forms was carried out in \cite{Gromov}.	
\end{remark}
\subsection{Locally conformally K\"{a}hler manifold}
In this section we state several equivalent definitions of the notion of a locally conformal K\"{a}hler (LCK) manifold.  Let $(M,J,g)$ be a complex manifold of $\dim_{\C}=n>1$, where $J$ denotes its complex structure and $g$ its Hermitian metric. Locally conformally K\"{a}hler (LCK) manifolds are, by definition, complex manifolds  admitting a K\"{a}hler covering with deck transformation acting by K\"{a}hler homotheties.  An equivalent definition is that there is an open cover $\{U_{i}\}$ of $M$ and a family $\{f_{i}\}$ of $C^{\infty}$ functions $f_{i}:U_{i}\rightarrow\mathbb{R}$ so that each local metric ${\color{red}g_{i}=\exp(-f_{i})g|_{U_{i}}}$ on $U_{i}$ is K\"{a}hlerian. Then the metrics ${\color{red}e^{-f_{i}}g_{i}}$ glue to a global metric whose associated $2$-form $\w$ satisfies the integrability condition  $d\w=\theta\wedge\w$,  thus being locally conformal with the K\"{a}hler metrics $g_{i}$. Here ${\color{red}\theta|_{U_{i}}=df_{i}}$. The closed $1$-form $\theta$ is called the \textbf{Lee form}. This gives another definition of an LCK structure, which will be used in this paper \cite{DO}.
\begin{definition}
Let $(M,J,g)$ be a complex Hermitian manifold, $\dim_{\C}M>1$, with 
$$d\w=\theta\wedge\w,$$
where $\theta$ is a closed 1-form. Then $M$ is called a \textbf{locally conformally K\"{a}hler (LCK)} manifold.
\end{definition}
A compact LCK manifold never admits a K\"{a}hler structure, unless the
cohomology class $\theta\in H^{1}(M)$ vanishes \cite{Vaisman}. 

If $(M,g)$ is a compact Riemannian manifold and $\pi:(\tilde{M},\tilde{g})\rightarrow (M,g)$ its universal covering. There is a very known result as follows. 
\begin{proposition}(\cite[Proposition 1]{JZ})\label{P3}
	If $\a$ is a closed $1$-form on {a compact Riemannian manifold $M$}, then the lifted $1$-form $\tilde{\a}:=\pi^{\ast}\a$ {on the universal covering space $\tilde{M}$} is $d$(sublinear).	
\end{proposition} 
Following Proposition \ref{P3}, it implies that the lifted Lee form $\pi^{\ast}\theta$ is $d$(sublinear) with respect to metric $\pi^{\ast}g$. In this article, we now introduce a class of LCK manifold as follows. 
\begin{definition}\label{D1}
	Let $(M,J,g)$ be a compact $2n$-dimensional LCK manifold with the Lee form $\theta$. We denote by $(\tilde{M},\tilde{J},\tilde{g})$ the universal covering space of $(M,J,g)$ and $\tilde{\theta}:=df$ the lifted Lee form. We call $(M,J,g)$ a $d$(bounded) LCK manifold, if $f$ is bounded.  
\end{definition}
\begin{example}
	A manifold $(M,J,g)$ called \textbf{globally conformal K\"{a}hler (GCK)} if there is a $C^{\infty}$ function $f:M \rightarrow \mathbb{R}$ such that the metric $e^{-f}g$ is K\"{a}hlerian, i.e, the Lee form $\theta=df$. Thus a closed, GCK manifold $(M,J,g)$ is a $d$(bounded) manifold because the lifted function $\pi^{\ast}f$ on $\tilde{M}$ is also bounded. One can see that if the first Betti number $b^{1}(M)=0$ (for example, $M$ is simply-connected) then $M$ is GCK. In particular, the universal covering space of a LCK manifold is GCK. 	
\end{example}
\begin{remark}
	A compact LCK (but not GCK) manifold cannot have strictly positive sectional curvature because (by a classical result of J. Synge \cite{Synge}) it would be simply connected and thus GCK.
\end{remark}
\section{Euler number of LCK manifold}
If $g_{1}=e^{f}g_{2}$ are two conformally equivalent Riemannian metric on a smooth $2n$-dimensional manifold $M$, then we have the equality, see \cite[Proposition 5.2]{Carron} : $$\mathcal{H}^{n}_{(2)}(M,g_{1})=\mathcal{H}^{n}_{(2)}(M,g_{2}).$$
The spaces of  $L^{2}$ harmonic forms depend only on the $L^{2}$ structures, hence if $g_{1}$ and $g_{2}$ are two Riemannian metrics on a manifold $M$ ($g_{1}$, $g_{2}$ need not to be complete) which are quasi-isometric that is for a certain constant 
$$C^{-1}g_{1}\leq g_{2}\leq Cg_{1},$$
then clearly the Hilbert spaces $\Om_{(2)}^{k}(M,g_{1})$ and $\Om_{(2)}^{k}(M,g_{2})$ are the same with equivalent norms. Hence the quotient spaces defining reduced $L^{2}$ cohomology are the same, that is
$$\mathcal{H}_{(2)}^{k}(M,g_{1})=\mathcal{H}_{(2)}^{k}(M,g_{2}).$$
In fact, the spaces $\mathcal{H}_{(2)}^{k}(M,g)$ are biLipschitz-homotopy invariants of $(M,g)$, see \cite[Proposition 3.1]{Carron}.
\begin{proposition}\label{P1}
	Let $(M,J,g)$ be a compact $2n$-dimensional $d$(bounded) LCK manifold with the Lee form $\theta$. If the sectional curvature of $M$ is non-positive (resp. negative), then there exist a $1$-form $\eta$ and a function $f$ on the universal covering space $(\tilde{M},\tilde{J},\tilde{g})$ such that\\
	(1) $$\tilde{\theta}=df,$$
	where $\tilde{\theta}$ is the lifted  $1$-form on $\tilde{M}$,\\
	(2) $$e^{-f}\tilde{\w}=d\eta,$$
	where $\tilde{\w}$ is the lifted K\"{a}hler form of $(\tilde{M},\tilde{g})$,\\
	(3) $$|\eta(x)|_{e^{-f}\tilde{g}}\leq c(\rho_{e^{-f}\tilde{g}}(x,x_{0})+1)$$
	 (resp. $|\eta(x)|_{e^{-f}\tilde{g}}\leq c$), where $c$ is uniform positive constant.
	 
	In particular, in non-positive sectional curvature case, for any $k\neq n$, 
	$$\mathcal{H}^{k}_{(2)}(\tilde{M},e^{-f}\tilde{g})=\{0\}$$
	and
	in negative sectional curvature case, 
	$$\mathcal{H}^{n}_{(2)}(\tilde{M},e^{-f}\tilde{g})\neq\{0\}.$$
\end{proposition}
\begin{proof}
	Let $\pi:(\tilde{M},\tilde{g})\rightarrow(M,g)$ be the universal covering map, $\w$ the K\"{a}hler form  and $\theta$ the Lee form on $M$. By the definition of $d$(bounded) LCK manifold, it implies that there exists a bounded function $f$ on $\tilde{M}$ such that the lifted Lee form $$\tilde{\theta}:=\pi^{\ast}\theta=df.$$ Noting that $$d(e^{-f}\tilde{\w})=-df\wedge e^{-f}\tilde{\w}+e^{-f}d\tilde{\w}=e^{-f}\tilde{\w}\wedge(\tilde{\theta}-df)=0$$ 
	and 
	$$|e^{-f}\tilde{\w}|_{\tilde{g}}\leq e^{-f}|\tilde{\w}|_{\tilde{g}}<\infty.$$  
	Hence the 2-form $e^{-f}\tilde{\w}$ is a bounded closed form on $(\tilde{M},\tilde{g})$. If the sectional curvature of $M$ is non-positive (resp. negative), it follows from Theorem \ref{T3} that 
	$$e^{-f}\tilde{\w}=d\eta$$
	 and  
	$$|\eta(x)|_{\tilde{g}}\leq c(1+\rho_{\tilde{g}}(x,x_{0})) (resp.\ |\eta(x)|_{\tilde{g}}\leq c),$$
	where $\eta$ is a one-form on $\tilde{M}$ and $c$ is a uniform positive constant. The metric on $\tilde{M}$ induced by the K\"{a}hler form $e^{-f}\tilde{\w}$ is $e^{-f}\tilde{g}$. The metrics $\tilde{g}$ and $e^{-f}\tilde{g}$ are quasi-isometric, since there exists a uniform positive constant $c'$ such that $$\frac{1}{c'}\rho_{\tilde{g}}(x,x_{0})\leq\rho_{e^{-f}\tilde{g}}(x,x_{0})\leq c'\rho_{\tilde{g}}(x,x_{0}).$$
	We then have (1) if the sectional curvature of $M$ is non-positive, then
	$$|\eta(x)|_{e^{-f}\tilde{g}}=e^{-\frac{f}{2}}|\eta(x)|_{\tilde{g}}\leq c''(\rho_{\tilde{g}}(x,x_{0})+1)\leq c''(\rho_{e^{-f}\tilde{g}}(x,x_{0})+1);$$
	(2) if the sectional curvature of $X$ is negative, then
	$$|\eta(x)|_{e^{-f}\tilde{g}}=e^{-\frac{f}{2}}|\eta(x)|_{\tilde{g}}\leq c''.$$
	Here $c''$ is a uniform positive constant only depends on $f$ and the metric $g$. Following from Theorem \ref{T5}, then for any $k\neq n$, $\mathcal{H}^{k}_{(2)}(\tilde{M},e^{-f}\tilde{g})=\{0\}$ (non-positive case) and  $\mathcal{H}^{n}_{(2)}(\tilde{M},e^{-f}\tilde{g})\neq\{0\}$ (negative case).
\end{proof}
We then have 
\begin{theorem}\label{T4}
	Let $M$ be a compact $2n$-dimensional $d$(bounded) LCK manifold. If the sectional curvature of $M$ is non-positive (resp. negative), then the Euler characteristic of $M$ satisfies the inequality $(-1)^{n}\chi(M)\geq$ (resp. $>$) $0$.		
\end{theorem}
\begin{proof}
	Nothing that the metrics $\tilde{g}$, $e^{-f}\tilde{g}$ are quasi-isometric. Then following Proposition \ref{P1}, in  non-positive sectional curvature case, we obtain that the spaces of  $L^{2}$-harmonic $k$-form on $(\tilde{M},\tilde{g})$ satisfy $$\mathcal{H}_{(2)}^{k}(\tilde{M},\tilde{g})=\{0\},\forall k\neq n$$
	and in negative sectional curvature case
	$$\mathcal{H}_{(2)}^{n}(\tilde{M},\tilde{g})\neq\{0\}.$$
	The Atiyah index theorem for covers \cite{Atiyah} then gives $(-1)^{n}\chi(M)\geq$ (resp. $>$) $0$.
\end{proof}
\begin{corollary}
	Let $M^{2n}$ be a compact $2n$-dimensional Riemannian manifold of non-positive (resp. negative) curvature. If $M^{2n}$ is homeomorphic to a GCK manifold, then the Euler characteristic of $M^{2n}$ satisfies the inequality $(-1)^{n}\chi(M^{2n})\geq$ (resp. $>$) $0$.	
\end{corollary}
\begin{proof}
The conclusion follows form Theorem \ref{T4} and the GCK manifold is $d$(bounded).
\end{proof}
In \cite{CX}, the authors shown that the property of $d$(sublinearity) has homotopy invariance.
\begin{lemma}(\cite[Lemma 3]{CX})\label{L1}
	Let $F:M_{1}\rightarrow M_{2}$ be a smooth homotopy equivalence between two compact Riemannian manifolds, $\pi:\tilde{M}_{i}\rightarrow M_{i}$ the universal covering maps for $i=1,2$. Then, for any closed differential form $\a$ on $M_{2}$, $\pi^{\ast}(\a)$ is $d$(sublinear) (resp. $d$(bounded)) on $\tilde{M}_{2}$ if the form $(F\circ\pi)^{\ast}(\a)$ is $d$(sublinear) (resp. $d$(bounded)) on $\tilde{M}_{1}$.
\end{lemma}
\begin{proof}[\textbf{Proof of Theorem \ref{T6}}]
	Since $F$ is a homotopy equivalence, there exists a smooth map $G:M_{2}\rightarrow M_{1}$ such that both $F\circ G$ and $G\circ F$ are homotopic to the identity maps. Clearly, the maps $F$ and $G$ can be lifted to the universal covering spaces. Let
	then $\tilde{F}:\tilde{M}_{1}\rightarrow\tilde{M}_{2}$ and $\tilde{G}:\tilde{M}_{2}\rightarrow \tilde{M}_{1}$ be the lifted maps, so that the following diagram commutes:
	
	\centerline{
		\xymatrix{
			\tilde{M}_{1}\ar[r]^-{\tilde{F}}\ar[d]_{\pi}& \tilde{M}_{2}\ar[r]^-{\tilde{G} }\ar[d]_{\pi}& \tilde{M}_{1}\ar[d]_{\pi} \\
			M_{1}\ar[r]^-{F} & M_{2}\ar[r]^-{G} & M_{1}
		}
	}
	
	Since $M_{1}$ is compact LCK manifold, the Lee form $\theta$ is bounded and $\pi$ is a local isometry, $\pi^{\ast}\theta$ is a bounded form on $\tilde{M}_{1}$. By the assumption, there is a bounded function $f$ on $\tilde{M}_{1}$ such that
	 $${\pi^{\ast}\theta=df.}$$
	{Following the idea in Lemma \ref{L1}, the assumption and the commutativity of the diagram imply that the form 
	\begin{equation*}
	(G\circ \pi)^{\ast}\theta=(\pi\circ\tilde{G})^{\ast}\theta=\tilde{G}^{\ast}(\pi^{\ast}\theta)=\tilde{G}^{\ast}df=d(\tilde{G}^{\ast}f)
	\end{equation*}
	is $d$(bounded).} Thus the form $e^{-\tilde{G}^{\ast}f}(G\circ \pi)^{\ast}\w$ on $\tilde{M}_{2}$ is closed because
	\begin{equation}\nonumber
	\begin{split}
	de^{-\tilde{G}^{\ast}f}(G\circ \pi)^{\ast}\w&=e^{-\tilde{G}^{\ast}f}\big{(}-d(\tilde{G}^{\ast}f)\wedge(G\circ\pi)^{\ast}\w+d(G\circ\pi)^{\ast}\w\big{)}\\
	&=e^{-\tilde{G}^{\ast}f}\big{(}{-(G\circ\pi)^{\ast}\theta\wedge(G\circ\pi)^{\ast}\w}+d(G\circ\pi)^{\ast}\w\big{)}\\
	&=e^{-\tilde{G}^{\ast}f}(G\circ\pi)^{\ast}{(-\theta\wedge\w+d\w)=0.}\\
	\end{split}
	\end{equation}
	Following Theorem \ref{T3}, it implies that $e^{-\tilde{G}^{\ast}f}(G\circ \pi)^{\ast}\w$ on $\tilde{M}_{2}$ is $d$(sublinear) (resp. $d$(bounded)). Thus the lifted K\"{a}hler form $ \tilde{F}^{\ast}\big{(}e^{-\tilde{G}^{\ast}f}(G\circ \pi)^{\ast}\w\big{)}$ on $\tilde{M}_{1}$ is $d$(sublinear) (resp. $d$(bounded)) as well. On the other hand, by the commutativity of the diagram 
	$$(G\circ F\circ \pi)^{\ast}\w=(G\circ\pi\circ\tilde{F})^{\ast}\w.$$
	We then have  
	$$ \tilde{F}^{\ast}\big{(}e^{-\tilde{G}^{\ast}f}(G\circ \pi)^{\ast}\w\big{)}=e^{-(\tilde{G}\circ\tilde{F})^{\ast}f} (G\circ\pi\circ\tilde{F})^{\ast}\w=e^{-(\tilde{G}\circ\tilde{F})^{\ast}f}(G\circ F\circ \pi)^{\ast}\w.$$
	Following Theorem \ref{T5}, it implies that the spaces of $L^{2}$-harmonic $k$-forms with respect to metric $e^{-(\tilde{G}\circ\tilde{F})^{\ast}f}(G\circ F\circ \pi)^{\ast}g_{1}$satisfies $$\mathcal{H}^{k}_{(2)}(\tilde{M}_{1},e^{-(\tilde{G}\circ\tilde{F})^{\ast}f}(G\circ F\circ \pi)^{\ast}g_{1})=\{0\}, \forall\ k\neq n,$$
	in non-positive sectional curvature  case, and $$\mathcal{H}^{k}_{(2)}(\tilde{M}_{1},e^{-(\tilde{G}\circ\tilde{F})^{\ast}f}(G\circ F\circ \pi)^{\ast}g_{1})\neq\{0\},$$
	in negative sectional curvature case.
	
	The metric on $M_{1}$ induced by $(G\circ F)^{\ast}\w$ is $(G\circ F)^{\ast}g_{1}$. The universal covering space of $(M_{1},(G\circ F)^{\ast}\w)$ is  $(\tilde{M}_{1},(G\circ F\circ\pi)^{\ast}\w)$. We denote by $\mathcal{H}^{k}_{(2)}(\tilde{M}_{1},(G\circ F\circ\pi)^{\ast}g_{1})$ the spaces of $L^{2}$-harmonic $k$-forms on $\Om^{k}_{(2)}(\tilde{M}_{1})$ with respect to metric $(G\circ F\circ\pi)^{\ast}g_{1}$. The Euler characteristic $\chi(M_{1})$ is a topology invariant. The Atiyah index theorem for covers \cite{Atiyah} then gives $$\chi(M_{1})=\sum_{k=0}(-1)^{k}h_{(2)}^{k}({M}_{1}),$$
	where $$h_{(2)}^{k}({M}_{1}):=\dim_{\Gamma}\mathcal{H}^{k}_{(2)}(\tilde{M}_{1},(G\circ F\circ\pi)^{\ast}g_{1}), \forall\ 0\leq k\leq n.$$
	Since the metrics $e^{-(\tilde{G}\circ\tilde{F})^{\ast}f}(G\circ F\circ \pi)^{\ast}g_{1}$, $(G\circ F\circ \pi)^{\ast}g_{1}$ are quasi-isometric, we obtain that $h^{k}_{(2)}({M}_{1})=0$, when $k\neq n$ in non-positive sectional curvature case and $h^{n}_{(2)}({M}_{1})=0$ in negative  sectional curvature case. Thus the Euler number of $M_{1}$ satisfies $(-1)^{n}\chi(M_{1})\geq$ (resp. $>$)  $0$.
\end{proof}

\section*{Acknowledgements}
We would like to thank Professor H.Y. Wang for drawing our attention to the LCK manifold and generously helpful suggestions about these. {We would also like to thank the anonymous referee for  careful reading of my manuscript and helpful comments.} This work is supported by Nature Science Foundation of China No. 11801539.

\end{document}